\documentclass{amsart}
\usepackage{amssymb,amsmath,amsthm}
\usepackage[shortlabels]{enumitem}

\newtheorem{theorem}{Theorem}[section]

\newtheorem{lemma}[theorem]{Lemma}

\newtheorem{definition}[theorem]{Definition}
\newtheorem{corollary}[theorem]{Corollary}

\theoremstyle{definition}

\newtheorem{example}[theorem]{Example}

\newcommand{\N}{\mathbb{N}}

\newcommand{\T}{\mathbb{T}}
\newcommand{\C}{\mathbb{C}}
\newcommand{\F}{\mathbb{F}}
\newcommand{\PV}{\mathcal{P}} 

\newcommand{\fA}{\mathfrak A}

\title{Selfless inclusions of $C^*$-algebras}

\author[B. Hayes]{Ben Hayes}
\address{\parbox{\linewidth}{Department of Mathematics, University of Virginia\\
		141 Cabell Drive, Kerchof Hall
		P.O. Box 400137,
		Charlottesville, VA 22904}}
\email{brh5c@virginia.edu}

\author[S. Kunnawalkam Elayavalli]{Srivatsav Kunnawalkam Elayavalli}
\address{\parbox{\linewidth}{Department of Mathematics, University of Maryland, College Park, \\
		4176 Campus Dr, College Park, MD 20742}}
\email{sriva@umd.edu}
\urladdr{https://sites.google.com/view/srivatsavke}

\author[G. Patchell]{Gregory Patchell}
\address{\parbox{\linewidth}{Mathematical Institute, University of Oxford, Andrew Wiles Building, \\ Radcliffe Observatory Quarter, Woodstock Road, Oxford, OX2 6GG, UK}}
\email{greg.patchell@maths.ox.ac.uk}
\urladdr{https://sites.google.com/view/gpatchel}

\author[L. Robert]{Leonel Robert}
\address{\parbox{\linewidth}{Department of Mathematics, University of Louisiana at Lafayette, \\
		217 Maxim Doucet Hall, 1401 Johnston Street, Lafayette, LA 70503, USA}}
\email{lrobert@louisiana.edu}

\begin{document}

\begin{abstract}
    We introduce and study a natural notion of selflessness for inclusions of C*-probability spaces, which in particular implies that all intermediate C*-algebras are selfless in the sense of \cite{robertselfless}. We identify natural sources of selfless inclusions in the realms of $\mathcal{Z}$-stable and free product C*-algebras. As an application of this, we prove selflessness for a new family of C*-probability spaces outside the regime of free products and group C*-algebras. These include the reduced free unitary compact quantum groups.
\end{abstract}
	\maketitle

\section{Introduction}

We continue the recent program on \emph{selflessness} in C*-algebra theory by extending this notion to inclusions and by identifying new examples of selfless C*-probability spaces arising from quantum groups. Selfless C*-probability spaces were introduced by the fourth author in \cite{robertselfless}. This framework unravels certain old problems in C*-algebra theory, as demonstrated in the recent work of Amrutam, Gao, Kunnawalkam Elayavalli, and Patchell \cite{amrutam2025strictcomparisonreducedgroup}. Since then, rapid progress has been made in the field, resulting in the development of several new techniques and applications \cite{elayavalli2025negativeresolutioncalgebraictarski, vigdorovich2025structuralpropertiesreducedcalgebras, HKER, raum2025strictcomparisontwistedgroup, ozawa2025proximalityselflessnessgroupcalgebras}. Notably, Ozawa very recently discovered a new approach to proving selflessness,
and used it to resolve several open questions 
concerning this notion
\cite{ozawa2025proximalityselflessnessgroupcalgebras}.

We introduce and study a natural notion of \emph{selfless inclusions} of C*-probability spaces, extending \cite{robertselfless}. 
We call an inclusion of C*-probability spaces $B\subset  (A,\rho)$ 	selfless if there exist a free ultrafilter $\omega$ and a C*-probability space $(C,\kappa)$, with $C\neq\C$, such that the first factor embedding $(B\subset A)\to (B*C\subset A*C)$ is existential (Definition \ref{def:ssinclusion}). 
Roughly put, while the notion of selfless C*-probability space asserts the existence  of nontrivial elements in an ultrapower $A^\omega$ that are free from $A$ in a nondegenerate way, the notion of selfless inclusion additionally keeps track of a C*-subalgebra $B$ of $A$ from which these freely independent elements are sourced.

If an inclusion $B\subset (A,\rho)$ is selfless,  then  all the intermediate C*-subalgebras $B\subset C\subset A$ are selfless. This gives a way to identify selfless C*-algebras by realizing them
as intermediate C*-subalgebras of selfless inclusions.
Note also that, since selfless C*-algebras are simple, every intermediate C*-subalgebra of a selfless inclusion
is simple. Thus, selfless inclusions are   C*-irreducible in the sense investigated by R{\o}rdam in  \cite{RordamIrred}. 

We prove several results on selfless inclusions that closely parallel results from \cite{robertselfless} on selfless C*-probability spaces. We also show that some of Ozawa's methods from \cite{ozawa2025proximalityselflessnessgroupcalgebras} can be straightforwardly adapted to prove selflessness for  purely infinite  and  for $\mathcal{Z}$-stable inclusions in the sense of Sarkowicz \cite{sarkowicz2025tensorially};   see Theorems \ref{thm: selfless-pi} and \ref{Z-stable selfless}.

In \cite{ozawa2025proximalityselflessnessgroupcalgebras}, Ozawa  defines the PHP (Powers--Haagerup--Pisier) property for groups and shows that groups with this property are C*-selfless. Here we formulate this property in the context of an inclusion and show that it again implies selflessness. Very recently, Flores, Klisse, {\'O} Cobhthaigh, and Pagliero proved that reduced free products satisfying an Avitzour-type condition are selfless \cite{fmmm2025selflessfreeprod}. We revisit their argument here, more explicitly framing it as a verification of the PHP property of an inclusion, and adapting it to suit our application to free complexifications   (Theorem \ref{alaavitzour}).

The free complexification construction was introduced by Banica in the context of compact quantum groups \cite{BanicaNote2008}. Some notable compact quantum groups, such as the free unitary compact quantum group, can be described as free complexifications \cite{BanicaU_n1997}. We show that, under Avitzour-type conditions, (reduced) free complexifications arise as intermediate C*-subalgebras of a selfless inclusion, and are therefore  selfless. This applies to the reduced free unitary quantum group:

\begin{theorem}[Theorem \ref{thm:cpct-qntm-unitry-slflss}]\label{thm:main1}
The reduced free unitary compact quantum group $A_u(n)$ is selfless for $n\geq 2$. 
\end{theorem}

By \cite[Theorem~3.1]{robertselfless} (see also \cite{BarlakSzabo2016seqsplit,FHLRTVW2021modelthryCstar}), this yields the following corollary: 

	\begin{corollary}
For $n\geq 2$,	$A_u(n)$ has stable rank one and strict comparison of positive elements by its unique trace.
	\end{corollary}


We also show, with the same strategy,  that the reduced half-liberated unitary quantum group $A_u^*(n)$ introduced in \cite{BDDstandardmodel} is selfless.
 

 \subsection*{Acknowledgements} This project was partially supported by the NSF grants DMS-2144739 (Hayes) and  DMS 2350049 (Kunnawalkam Elayavalli). The
third author was supported in part by the Engineering and Physical Sciences Research Council (UK), grant EP/X026647/1. This work has received funding from the European Research Council (ERC) under the European Union’s Horizon 2020 research and innovation programme (Grant agreement No. 850930). We thank S. White for helpful feedback and S. Raum for clarifying remarks around the half-liberated unitary quantum group.

\subsection*{Open Access and Data Statement} For the purpose of Open Access, the authors have applied a CC BY public copyright license to any Author Accepted Manuscript (AAM) version arising from this submission. Data sharing is not applicable to this article as no new data were created or analyzed in this work.

	\section{Selfless Inclusions}
	
	We work in the setting of C*-probability spaces $(A,\rho)$; i.e., 
    unital C*-algebras endowed with a distinguished state.
We always assume that the state induces a faithful GNS representation of the C*-algebra, with the exception of ultrapower states. Morphisms between C*-probability spaces are unital *-homomorphisms preserving the respective states.
We often omit reference to the state of a C*-probability space,  e.g., in the formation of reduced free products,
 if it can be inferred from the context. We note however that all free products below are reduced free products with respect to the given states. 

Given a C*-probability space $(A,\rho)$ and an ultrafilter $\omega$ (over some set), we obtain the ultrapower C*-probability space $(A^\omega,\rho^\omega)$, where $A^\omega$ is the C*-algebra ultrapower of $A$ and $\rho^\omega$ the limit state (along $\omega$) induced by $\rho$.

	By an inclusion of C*-probability spaces we understand a pair of C*-probability spaces  $(A,\rho)$ and $(B,\rho|_B)$, where $B$ is a unital C*-subalgebra of $A$ and the state on $B$ is the restriction of the state on $A$. To denote it, we write $B\subset A$ or $B\subset (A,\rho)$. By an embedding of $B_1\subset (A_1,\rho_1)$ in $B_2\subset (A_2,\rho_2)$ we understand a unital embedding of C*-algebras $\theta\colon A_1\to A_2$ that is state preserving, i.e. $\rho_2\theta=\rho_1$, and such that $\theta(B_1)\subseteq B_2$.

Let $\theta\colon (B_1\subset (A_1,\rho_1)) \to (B_2\subset (A_2,\rho_2))$ be an embedding of inclusions. We call $\theta$ existential if there exists an ultrafilter $\omega$ and an embedding 
	\[
	\sigma\colon (B_2\subset (A_2,\rho_2))\to (B_1^\omega\subset (A_1^\omega,\rho_1^\omega))
	\] such that
	$\sigma\circ\theta$ agrees with the diagonal inclusion of $A_1$ in $A_1^\omega$. Equivalently, there exists a C*-algebra
	embedding $\sigma\colon A_2\to A_1^\omega$ such that $\sigma$ is state preserving, $\sigma\circ \theta$ agrees with the diagonal inclusion, and
	 $\sigma(B_2)\subseteq B_1^\omega$. Note that in this case the
induced embeddings	$\theta\colon A_1\to A_2$ and $\theta|_{B_1}\colon B_1\to B_2$ are both existential embeddings of C*-probability spaces. 
See \cite[Section 1]{robertselfless} for a discussion of existential embeddings of C*-probability spaces.

The following lemma is elementary.

	\begin{lemma}\label{lem:easyexistential}
		Let $\theta_1$ and $\theta_2$ be embeddings between inclusions of C*-probability spaces such that $\theta_2\circ\theta_1$ is well-defined.
	\begin{enumerate}[(i)]
		\item 	If $\theta_1$ and $\theta_2$ are existential, then $\theta_2\circ\theta_1$ is also existential. 
		\item
		If $\theta_2\circ\theta_1$ is existential, then so is $\theta_1$.
		\end{enumerate}
	\end{lemma}
	
We also have the following corollary of ``free exactness" (cf. \cite[Corollary~1.9]{robertselfless}).

\begin{theorem}\label{thm: free-prod-exstl-embed}
A reduced free product of existential embeddings of inclusions  is again existential. That is, given inclusions $B_i\subset (A_i,\rho_i)$, 
$D_i\subset (C_i,\psi_i)$ for $i=1,2$, and existential embeddings
\[
\theta_i\colon (B_i\subset (A_i,\rho_i))\to (D_i\subset (C_i,\psi_i))
\]
for $i=1,2$, the induced map
\[
\theta_1 * \theta_2 \colon (B_1*B_2 \subset A_1*A_2) \to (D_1*D_2 \subset C_1*C_2)
\]
is existential.
\end{theorem}

	\begin{definition}\label{def:ssinclusion}
		We call an inclusion of C*-probability spaces $B\subset  (A,\rho)$ 
		selfless if there exist a free ultrafilter $\omega$ and a C*-probability space $(C,\kappa)$, with $C\neq\C$, such that the first factor embedding of inclusions
		\[
		\theta\colon (B\subset A)\to (B*C\subset A*C)
		\] is existential.
		More concretely, $B\subset  A$ is selfless if there exist an ultrafilter $\omega$, a C*-probability space $(C,\kappa)$ with 	
		$C\neq\C$, and a state-preserving embedding of C*-algebras 
		$\sigma\colon A*C\to A^\omega$, such that $\sigma\circ\theta$ agrees with the diagonal inclusion  and 
        $\sigma(B*C)\subseteq B^\omega$ (equivalently, $\sigma(C)\subseteq B^\omega$).
	\end{definition}

Notice that a C*-probability space $(A,\rho)$ is selfless in the sense of \cite{robertselfless} if and only if $A\subset  (A,\rho)$ is a selfless inclusion, and that if $B\subset (A,\rho)$ is a selfless inclusion, then both $(A,\rho)$ and $(B,\rho|_B)$ are selfless C*-probability spaces.

\begin{lemma}
	Let $\phi \colon (B_1\subset A_1)\to (B_2\subset A_2)$ be an existential embedding of inclusions and suppose that $B_2\subset A_2$
	is selfless. Then so is $B_1\subset A_1$.
\end{lemma}

\begin{proof}
By assumption, there exist $C\neq \C$ such that the first factor embedding  $\theta\colon (B_2\subset A_2)\to (B_2*C\subset A_2*C)$ is existential. Thus, the embedding 
\[
\theta\circ \phi\colon (B_1\subset A_1)\to (B_2*C\subset A_2*C)
\] 
is existential (Lemma \ref{lem:easyexistential} (i)). Since $\theta\circ\phi$ ranges in $B_1*C\subset A_1*C$, its corestriction to  $B_1*C\subset A_1*C$ is also existential (Lemma \ref{lem:easyexistential} (ii)). 
This shows that $B_1\subset A_1$ is selfless.
\end{proof}

	In the following theorem we sometimes omit reference to the state, but the reader should bear in mind that we work in
	the category of C*-probability spaces and state preserving unital *-homomorphisms.

	\begin{theorem}\label{thm:ssinclusions}
		Let $B\subset  (A,\rho)$ be a selfless inclusion of C*-probability spaces. The following are true:
		\begin{enumerate}[(i)]  
			\item
			If $B\subset C\subset A$, then $B\subset C$ and $C\subset A$ are selfless. In particular, $C$ is selfless. 
			
			\item  Given an inclusion of C*-probability spaces 
			$D_1\subset D_2$, the reduced free product inclusion $B*D_1\subset A*D_2$ is selfless. 
			

			\item The  C*-probability space $(C,\kappa)$ witnessing selflessness of $B\subset  (A,\rho)$
			in Definition \ref{def:ssinclusion} may be chosen to be $(C_r^*(\mathbb F_\infty),\tau)$.

			
			\item Choose any $2/\sqrt{5}<\alpha<1$. Then for every $a\in A$ there exist unitaries $u_1,\ldots,u_5\in B$ such that 
            \begin{equation}\label{5Dixmier}
            \Big\|\frac15\sum_{i=1}^5 u_iau_i^* - \rho(a)1\Big\|\leq \alpha\|a\|.
            \end{equation}
            Consequently, the inclusion $B\subset A$ has the relative Dixmier property with respect  to $\rho$; i.e.,  
			$\rho(a)1\in \overline{\operatorname{co}}\{u a u^*:u\in U(B)\}$ for all $a\in A$. 	

            \item Assume that $\rho$ is faithful.
            Then  $\rho$ can be $B$-excised in the following sense: there exists a net $(b_\lambda)_\lambda$ of positive elements of norm 1
            in $B$  such that 
            $\|b_{\lambda}^{\frac12}ab_{\lambda}^{\frac12}-\rho(a)b_\lambda\|\to 0$ and $\rho(ab_\lambda a^*)\to 0$ for all $a\in A$. 
		\end{enumerate}
	\end{theorem}
	
	\begin{proof}
		(i) This is straightforward.
		
		(ii) Since $B\subset A$ is selfless, the first factor embedding
		$(B\subset A)\to (B*C\subset A*C)$ is existential for some nontrivial $C$.  Since the reduced free product of 
		existential embeddings (of inclusions) is again existential (Theorem \ref{thm: free-prod-exstl-embed}), and the identity embedding $(D_1\subset D_2)\to (D_1\subset D_2)$ is trivially existential,  
	the embedding of inclusions	
	\[
	(B*D_1\subset A*D_2)\to (B*C*D_1\subset A*C*D_2)
	\] 
	is existential. It follows that $B*D_1\subset A*D_2$ is a selfless inclusion.

		
		(iii) Let $C\neq \C$ be such that the first factor embedding $(B\subset A) \to (B*C\subset A*C)$ is existential.
		Taking reduced free product with the identity $(C\subset C)\to (C\subset C)$, we get that 
        \[(B*C\subset A*C) \to (B*C*C\subset A*C*C)\]
		is existential, and so 
        \[(B\subset A)\to  (B*C*C\subset A*C*C)\]
        is existential. Applying this argument 
		repeatedly, we get that 
        \[
        (B\subset A)\to  (B*C^{*n}\subset A*C^{*n})
        \]
        is existential. For large enough $n$,
		$C^{*n}$ contains an embedding of $C_r^*(\mathbb F_\infty)$ 
        (by \cite[Lemma 2.5]{robertselfless}). Thus,
        \[
        (B\subset A)\to (B*C_r^*(\mathbb F_\infty)\subset A*C_r^*(\mathbb F_\infty))
        \]
        is existential, as desired.
		
		
		(iv) Avitzour's proof of simplicity of the reduced free product  $A*C_r^*(\mathbb F_\infty)$ finds for every $a\in A\subseteq A*C_r^*(\mathbb F_\infty)$ unitaries $u_1,u_2,\ldots,u_5\in C_r^*(\mathbb F_\infty)$ such that \eqref{5Dixmier} holds with $\alpha=\frac{2}{\sqrt 5}$ (see \cite[Proposition 3.1]{Avitzour}).
         Applying $\sigma\colon A*C_r^*(\mathbb F_\infty)\to A^\omega$, and using that	the image of $C_r^*(\mathbb F_\infty)$ is contained in $B^\omega$, we get \eqref{5Dixmier} with $\alpha=\frac2{\sqrt 5}$ and unitaries in $B^\omega$. Lifting these unitaries to sequences of unitaries in $B$, we obtain  \eqref{5Dixmier} with any  $\frac2{\sqrt{5}}<\alpha<1$ and unitaries in $B$.

        (v) It will suffice to show that for a set $\mathcal S$ whose linear span  is  dense in $A$, and for every finite set $F\subseteq \mathcal S$ and $\epsilon>0$, there exists $b\in B_+$ of norm 1 such that $\|b^{\frac12}ab^{\frac12}-\rho(a)b\|<\epsilon$ for all $a\in F$. We will show this with 
        \[
        \mathcal S=\{1\}\cup \{a\in A_+:\|a\|=1\hbox{ and $ad=0$ for some nonzero $d\in A_+$}\}.
        \]

        Given $a\in A_+$, define $d_\rho(a)=\lim_n \rho(a^{\frac 1n})$.
        Note that if $ad=0$ for some  nonzero $d\in A_+$, then $d_{\rho}(a)+d_{\rho}(d)=d_\rho(a+d)\leq 1$. On the other hand, since $\rho$ is faithful, $d_\rho(d)>0 $. It follows that if $a\in \mathcal S$, then $d_\rho(a)<1$.
 
        Let $\epsilon>0$ and $F\subseteq \mathcal S$. 
        Choose $\delta>0$  such that $d_\rho(a)<1-\delta$
        for all $a\in F$. Choose a $c\in C_r^*(\mathbb F_\infty)_+$ of norm 1 such that $d_\tau(c)<\min(\delta, \epsilon^2)$, which
        exists since $C_r^*(\mathbb F_\infty)$ is simple and nonelementary. Then,
        in $A*C_r^*(\mathbb F_\infty)$ we have
        \[
        \|c^{\frac12}(a-\rho(a))c^{\frac12}\| <\epsilon
        \]
        for all $a\in \{1\}\cup F$, by \cite[Lemma 8.1]{robertselfless}. Also, by the free independence of $A$ and $c$, 
        we have
        \[
        (\rho*\tau)(aca)=\rho(a^2)\tau(c)<\epsilon^2,
        \]
        for all $a\in \{1\}\cup F$.
        Let $\sigma\colon A*C_r^*(\mathbb F_\infty)\to A^\omega$ be such that
        $\sigma\circ\theta$ is the diagonal inclusion of $A$ in $A^\omega$ and $\sigma(C_r^*(\mathbb F_\infty))\subseteq B^\omega$.
        Set  $b=\sigma(c)\in B^\omega$, and lift  $b$ to $(b_i)_i$, where each $b_i$ is  a positive element of norm $1$.
        Then $\|b_i^{\frac12}(a-\rho(a))b_i^{\frac12}\|<\epsilon$ and $\rho(ab_ia)<\epsilon^2$
        for all $a\in \{1\}\cup F$ and $\omega$-almost all $i$.
	\end{proof}

	An inclusion $B\subset (A,\rho)$ is said to have relative excision, in the sense of R{\o}rdam, 
	if $\rho$ can be excised (in $A$) with a net of positive contractions coming from $B$ \cite[Definition~3.14]{RordamIrred}. 
	(We also say in this case that $\rho$ is $B$-excised.) 
    R{\o}rdam showed that if $B$
	is simple, unital, and infinite dimensional, and $\Gamma$ a discrete group with an outer action on $B$, then $B\subset B\rtimes_r \Gamma$ has relative excision 
	with respect to every state that factors through the conditional expectation onto $B$ \cite[Lemma~5.7]{RordamIrred}. Theorem \ref{thm:ssinclusions} (v) implies that if the inclusion $B\subset (A,\rho)$ is selfless and 
   $\rho$ is faithful, then $B\subset (A,\rho)$ has relative excision.

We require the following lemma due to Ozawa.

\begin{lemma}\label{lem: univ-free-prod}(\cite[Lemma~12]{ozawa2025proximalityselflessnessgroupcalgebras}).
	Let $(A,\rho)$ be a $C^*$-probability space and $(\mathcal T,\omega)$ the Toeplitz probability space (the universal $C^*$-algebra generated by an isometry). Then $A * \mathcal T$ is the universal $C^*$-algebra generated by $A$ and an isometry $T$ (the generator of the Toeplitz algebra) that satisfies $T^*aT = \rho(a)1$ for $a\in A$. The free product state $\psi$ on $A*\mathcal T$ is the unique state that satisfies $\psi|_A = \rho $ and $\psi(aTT^*a^*)=0$ for all $a\in A$.
\end{lemma}

\begin{proof}
	See Example 4.6.11 and Exercise 4.8.1 in \cite{brown2008textrm}.
\end{proof}

	\begin{theorem}\label{thm: selfless-pi}
		Let $B\subset  (A,\rho)$ be an inclusion 
        such that $B$ is simple and purely infinite
        and $\rho$ can be $B$-excised by a net  of positive contractions $(b_\lambda)$ also satisfying that  
        $\rho(ab_\lambda a^*)\to 0$ for all $a\in A$. Then  $B\subset (A,\rho)$ is selfless.
	\end{theorem}	
	
	\begin{proof}
		We follow Ozawa's proof that simple purely infinite C*-algebras are selfless \cite[Theorem~3]{ozawa2025proximalityselflessnessgroupcalgebras}.

        Using that $B$ has real rank zero, we can assume that $b_\lambda$ has finite spectrum (and norm 1) for all $\lambda$. Then, with  $p_\lambda$ the projection $1_{\{1\}}(b_{\lambda})$ obtained by functional calculus, we have $b_{\lambda}^\frac 12 p_\lambda=p_\lambda$. We deduce that $(p_\lambda)$
        still excises $\rho$ and $\rho(ap_\lambda a^*)\to 0$ for all $a\in A$.

        Using now that $B$ is simple and  purely infinite, we choose isometries $(T_\lambda)$ in $B$ such that  $T_\lambda T_{\lambda}^*\leq p_\lambda$ for all $\lambda$. The net $(T_\lambda)$ still excises $\rho$ and satisfies that $\rho(aT_\lambda T_\lambda^* a^*)\to 0$ for all $a$. Choose a free ultrafilter $\omega$ over the index set for $\lambda$ containing the Fr{\'e}chet filter. Letting $T\in B^\omega$ be the image of $(T_\lambda)$, we obtain an isometry such that
        $T^*aT=\rho(a)1$ and $\rho(aTT^*a^*)=0$ for all $a\in A$.
        It follows by Lemma \ref{lem: univ-free-prod} that there exists a state preserving isomorphism $\sigma\colon A*\mathcal T\to C^*(A,T)$ that restricts to the identity on $A$ (with $\mathcal T$ denoting the Toeplitz algebra endowed with the vacuum state). Since $\sigma(\mathcal T)\subseteq B^\omega$, this shows that $B\subset (A,\rho)$ is a selfless inclusion. 
    	\end{proof}	

Note that the condition $\rho(ab_\lambda a^*)\to 0$ for all $a\in A$ in the theorem above is necessary to reach the conclusion when $\rho$ is faithful, by Theorem \ref{thm:ssinclusions} (v), and may always be necessary. On the other hand, if $\rho$ is a diffuse state, then any sequence of positive contractions excising it must converge weakly to 0 \cite[Corollary 2.15]{AAPdiffuse}. Thus, in this case, the condition $\rho(ab_\lambda a^*)\to 0$ for all $a\in A$ holds automatically.
    
	Let $\mathcal Z$ denote the Jiang--Su C*-algebra. In \cite{sarkowicz2025tensorially}, Sarkowicz  calls an inclusion $B\subset A$  
    $\mathcal Z$-stable if it is  isomorphic to the tensor inclusion $B\otimes \mathcal Z\subset A\otimes \mathcal Z$. 

  \begin{lemma}
  If $B\subset A$ is $\mathcal Z$-stable, then 
  the first factor embedding 
  \[
  \theta\colon (B\subset A)\to (B\otimes \mathcal Z\subset A\otimes \mathcal Z)
  \]
  is existential. That is, there exists an embedding $\sigma\colon A\otimes \mathcal Z\to A^\omega$, for some ultrafilter $\omega$, such that $\sigma\circ\theta$ agrees with the diagonal inclusion and $\sigma(B\otimes \mathcal Z)\subseteq B^\omega$.
  \end{lemma}
  \begin{proof}
This is well known for a single $\mathcal Z$-stable C*-algebra. We now keep track of a C*-subalgebra.    
    Set $D=\mathcal Z^{\infty\otimes}$ and $D_n=\mathcal Z^{n\otimes}\otimes 1\otimes \cdots\subseteq D$. Choose ``permutation" isomorphisms $\sigma_n\colon D\otimes \mathcal Z\to D$ for $n=1,2,\ldots$ that are the identity on $D_n\otimes 1$. 
         Now tensor with $A$ to get isomorphisms $\mathrm{id}_A\otimes \sigma_n$ from $(A\otimes D)\otimes \mathcal Z$ to $A\otimes D$. Choose a free ultrafilter $\omega$ on $\N$ and use $(\sigma_n)_n$ to
    define $\sigma\colon (A\otimes D)\otimes \mathcal Z\to (A\otimes D)^\omega$ such that $\sigma\circ\theta$ is the diagonal inclusion.  
    We have shown that $\theta\colon A\otimes D\to (A\otimes D)\otimes \mathcal Z$ is existential. Observe now that the image of $(1\otimes 1_D)\otimes \mathcal Z$ under $\sigma$  is contained in $(1\otimes D)^\omega$. It follows that for any $1\in B\subseteq A$, the first factor embedding  
    \[
    (B\otimes D\subset A\otimes D)\to ((B\otimes D)\otimes \mathcal Z\subset (A\otimes D)\otimes \mathcal Z)
    \]
    is existential. But $D\cong \mathcal Z$, and  $B\subset A$ is isomorphic to $B\otimes \mathcal Z\subset A\otimes \mathcal Z$ by assumption. The lemma thus follows.
\end{proof}
    
	\begin{theorem}\label{Z-stable selfless}
		Let $B\subset (A,\rho)$ be a $\mathcal Z$-stable inclusion of C*-probability spaces 
that has the relative Dixmier property, where $A$ is exact and  
        $\rho$ is $U(B)$-invariant. Then $B\subset (A,\rho)$ is a selfless inclusion.	
	\end{theorem}	
	
	\begin{proof}
		We follow Ozawa's proof \cite[Theorem~3]{ozawa2025proximalityselflessnessgroupcalgebras} for the case $B=A$. 
        
        Let $(\mathcal C,\tau )$ denote the free semicircular
system with countably many generators. 
By the previous lemma, there exists 
$\sigma_1\colon A\otimes \mathcal Z\to A^\omega$ 
that is the diagonal inclusion on $A\otimes 1$ and maps
$B\otimes \mathcal Z$ to $B^\omega$. (Note: from the proof of the lemma, $\omega$ can be any free ultrafilter over $\N$.)  By \cite[Theorem~4.1]{ozawa2023amenability}, 
$\mathcal C$ embeds in $\mathcal Z^\omega$. Tensoring with $A$ and using its exactness (see e.g. \cite[the proof of Proposition 3.11(ii)]{Hayes3}, \cite[Lemma in the appendix]{Male}, \cite[Lemma 2.18]{vHCDM}), we obtain an embedding of $A\otimes \mathcal C$ in
    $(A\otimes \mathcal Z)^\omega$. Composing it with $\sigma_1^\omega$,  we get $\sigma\colon A\otimes \mathcal C\to (A^\omega)^\omega\cong A^{\omega\otimes\omega}$ such that $A\ni a\mapsto \sigma(a\otimes 1)$
	agrees with the diagonal inclusion and  $1\otimes \mathcal C$ is mapped into $B^{\omega\otimes\omega}$. Relabel $\omega\otimes\omega$ as $\omega$ (as it will be irrelevant which specific ultrafilter we are working with), so as to now have $\sigma\colon A\otimes\mathcal C\to A^\omega$. Let us argue that
     $\sigma$ is state preserving; i.e., $\rho^\omega\sigma=\rho\otimes \tau$. For $b\in A'\cap B^\omega$, the functional $A\ni a\mapsto \rho^\omega(ab)$ is $U(B)$ invariant. By the relative Dixmier property of $B\subset A$, any such functional is a scalar multiple of $\rho$. Hence, $\rho^\omega(ab)=\rho(a)\rho^\omega(b)$ for all $a\in A$, $b\in A'\cap B^\omega$, which implies that 
     \[
     \rho^\omega\sigma(a\otimes c)=\rho(a)\rho^\omega(\sigma(1\otimes c))=
     \rho(a)\tau(c), 
    \]
    for all $a\in A$ and $c\in \mathcal C$ (where we have used the uniqueness of trace in $\mathcal C$). Thus, $\sigma$  is state  preserving.  This means that  the first factor embedding 
	\[
(B\subset (A,\rho))\to (B\otimes \mathcal C\subset (A\otimes \mathcal C, \rho\otimes \tau))
	\]
	is an existential embedding of inclusions of C*-probability spaces. Trivially, the diagonal embedding 
	\[
	(B\otimes \mathcal C\subset A\otimes \mathcal C)\to ((B\otimes \mathcal C)^\omega\subset (A\otimes \mathcal C)^\omega)
	\]
is also existential. Composing them, we obtain that the embedding of $B\subset (A,\rho)$
in  
\[
(B\otimes \mathcal C)^\omega\subset ((A\otimes \mathcal C)^\omega, (\rho\otimes\tau)^\omega)
\]
is existential.
We will be done once we have shown that this embedding can be factored through the first factor embedding of $B\subset A$
in  $B*\mathcal C_1\subset A*\mathcal C_1$,
where $\mathcal C_1$ is the C*-probability space generated by a single semicircular element. Put differently, we will be done once we have found a semicircular element $s\in (B\otimes \mathcal C)^\omega$ such that $C^*(A,s)\cong A*\mathcal C_1$.

The proof continues verbatim along the lines of Ozawa's proof, so we will only indicate the main steps. Using
the relative Dixmier property of $B\subset A$, we choose 
collections of unitaries $(u_{i,k})_{k=1}^{N_i}$ in $B$ such that 
the unitary mixing operators
\[
A\ni a\mapsto \frac{1}{N_i}\sum_{k=1}^{N_i} u_{i,k}au_{i,k}^*
\]
converge pointwise to $A\ni a\mapsto \rho(a)1$. We arrange for
$(u_{i,k})_{k=1}^{N_i}$ to be closed under conjugation; i.e., 
$(u_{i,k}^*)_{k=1}^{N_i}$ and $(u_{i,k})_{k=1}^{N_i}$ agree up to a permutation. A similar argument to the previous paragraph shows that the embedding of inclusions
\[
(B\subset A) \to (B\otimes C_r^*(\F_\infty) \subset A\otimes C_r^*(\F_\infty))
\]is existential. Thus, by replacing the $u_{i,k}$ with $u_{i,k}\otimes t_k$ (where the $t_k$ are the generators of $\F_\infty$), using the resulting free independence, and applying Voiculescu's inequality \cite[Proposition~7.4]{junge2005picuineq}, we may also arrange for 
\[
\frac{1}{N_i}\sum_{k=1}^{N_i} u_{i,k}au_{i,k}\to 0
\]
for all $a\in A$ with $\rho(a)=0$.

Let  $l_k\in \mathcal O_\infty$ for $k=1,\ldots$ be the isometries generating $\mathcal O_\infty$. Identify 
$\mathcal C$ with the C*-subalgebra generated by $\{(l_k+l_k^*)/2:k=1,\ldots\}$ and $\tau$ with the restriction of the vacuum state on $\mathcal O_\infty$ to $\mathcal C$.
Define $T_i\in B\otimes \mathcal O_\infty$ as
\[
T_i=\frac{1}{\sqrt{N_i}}\sum_{k=1}^{N_i} (u_{i,k}+u_{i,k}^*)\otimes l_k, 
\]
and  
\[
s_i=\frac{1}{2}(T_i+T_i^*)\in B\otimes \mathcal C.
\]
Let $s=(s_i)_i\in (B\otimes \mathcal C)^\omega$.
By Ozawa's proof  $s\in (B\otimes \mathcal C)^\omega$ is a semicircular element such that $C^*(A,s)\cong A*\mathcal C_1$.
	\end{proof}

 \begin{example}
 Let $\mathcal O_2$ denote the Cuntz algebra
 generated by two isometries, and let $M_{2^\infty}\subseteq \mathcal O_2$ denote the canonically embedded UHF algebra \cite[4.2]{RordamStormer-EMS126}. 
 Let $\rho$ be the state on $\mathcal O_2$ obtained via the expectation  $E\colon \mathcal O_2 \to M_{2^\infty}$ followed by the trace  $\tau$ on $ M_{2^\infty}$. Then, by Theorem~\ref{Z-stable selfless}, the inclusion $M_{2^\infty}\subset (\mathcal O_2,\rho)$ is selfless. (That is, the selflessness of $(\mathcal O_2,\rho)$ can be certified by Haar unitaries  coming from $M_{2^\infty}$.)
 Indeed, by \cite[Example 6.8]{sarkowicz2025tensorially}, the inclusion  $M_{2^\infty}\subset \mathcal O_2$ is $\mathcal Z$-stable. In the lemma below we show that it has the relative Dixmier property.
The other conditions of Theorem~\ref{Z-stable selfless} can be routinely checked.
  \end{example}
	
\begin{lemma}
The inclusion $M_{2^\infty}\subset (\mathcal O_2,\rho)$ has the relative Dixmier property.
\end{lemma}

\begin{proof}
Let $s_1,s_2$ denote the isometries that generate $\mathcal O_2$.		
It suffices to show that a dense set of elements  $a\in \mathcal O_2$ can be $U(M_{2^\infty})$-averaged to $(\tau\circ E)(a)$. Thus, we may assume that
\[
a=\sum_{k=1}^{m} (s_1^*)^kb_{-k}+b_0+\sum_{k=1}^mb_ks_1^k
\]
where $b_k\in M_{2^\infty}$ for all $k$. Consider a term $bs_1^k$. For $u\in U(M_{2^\infty})$, we have 
\[
u^{*} bs_1^ku=(u^*b\lambda^k(u)) s_1^k
\]
where $\lambda\colon M_{2^\infty}\to M_{2^\infty}$ is the Bernoulli right-shift endomorphism. We thus see that the set $M_{2^\infty}s_1^k$ is convex and invariant under conjugation by unitaries in $M_{2^\infty}$. This holds similarly for the terms $(s_1^*)^kb$.  
Therefore, by a process of successive averaging, we can consider each of the terms of the above sum separately. Since $M_{2^\infty}$ has the Dixmier property, $b_0$ can be $U(M_{2^\infty})$-averaged to $\tau(b_0)$. Let us show that the remaining terms can be $U(M_{2^\infty})$-averaged to 0.  Equivalently, we must show that each $b\in M_{2^\infty}$ can be averaged to 0 in the following sense:
\[
T_i(b)=\frac{1}{N_i}\sum_{j=1}^{N_i}u_{i,j}^*b\lambda^k(u_{i,j})\to 0,
\]
for $u_{i,j}\in U(M_{2^\infty})$. This is precisely the type of unitary averaging relative to an automorphism studied by R{\o}rdam in \cite{Rordam-Avg2023}, except $\lambda^k$ is here an endomorphism rather than an automorphism.

Fix $k\in \N$. Set $M_{2^n}:=M_2(\C)^{\otimes n}\otimes 1\otimes 1\cdots\subset M_{2^\infty}$ for all $n$. Suppose without loss of generality that $b\in M_{2^n}$ for some $n$. Notice that if we choose the unitaries $u_{i,j}$ of the form $1^{\otimes n}\otimes u_{i,j}'$, so that they belong to $M_{2^n}'\cap M_{2^\infty}$, then
$T_i(b)=bT_i(1)$.  We thus see that it suffices to consider the case $b=1$. 

Let $\epsilon>0$. By \cite[Proposition 5.1.3]{RordamStormer-EMS126}, there exist $N\in \N$ and projections $p_1,\ldots,p_{2^k}\in M_{2^N}$ such that $1=\sum_{i=1}^{2^k} p_i$ and   $\|\lambda^k(p_i)-p_{i+k}\|<\epsilon$ for all $i$ (with indices taken modulo $2^k$).  For $\omega=(\omega_i)_{i=1}^{2^k}\in \mathbb T^{2^k}$ define
$u(\omega)=\sum_{i=1}^{2^k} \omega_ip_i$. Then
\[
y:=\int_{\mathbb T^{2^k}}u(\omega)^*\lambda^k(u(\omega))d\omega =\sum_{i=1}^{2^k} p_i\lambda^k(p_i) 
\]
(cf. \cite[Lemma 2.13]{Rordam-Avg2023}). By the orthogonality of the projections $(p_i)_i$, we have
$\|y\|=\max_i \|p_i\lambda^k(p_i)\|<\epsilon$, where we have used that   $\lambda^k(p_i)\approx_{\epsilon} p_{j}$ for $j\neq i$.
Since $y$ is a limit of averages of the form $T_i(1)$, and $\epsilon$ is arbitrary, we are done.
\end{proof}

\section{The PHP property and reduced free products}
Following Ozawa in the group case, we define property PHP (Powers--Haagerup--Pisier) for a $C^*$-probability space $(A,\rho)$ represented faithfully on a Hilbert space $H$, and more generally, for an inclusion $B\subset (A,\rho)$.

\begin{definition}(Cf. \cite[Section 8]{ozawa2025proximalityselflessnessgroupcalgebras}.)
	Let $(A,\rho)$ be a C*-probability space and assume that $A\subseteq \mathbb B(H)$. 
	Let $B\subset A$ be a $C^*$-subalgebra. Let us say that the inclusion $B\subset (A,\rho)$ has the PHP property if for all finite sets  $F\subset \ker(\rho)$, all $\epsilon > 0$, and all $n\in \N$, there exist $u_i,P_{i},P_i^+$, for $i=1,\ldots,n$ such that  $u_i\in  U(B)$, 
	$P_i,P_i^+\in \mathbb B(H)$ are projections with $P_i^+\leq P_i$,  and the following hold:
	\begin{enumerate}
		\item $(P_i)_{i=1}^n$ is a pairwise orthogonal family of projections;
		\item $u_iP_i^+u_i^*\leq P_i^+$ and $u_i(1-P_i)u_i^* \leq  P_i^+$  for all $i$;
		\item $\|P_ixP_j\|<\epsilon$ for all $x\in F$ and all $i,j$.
	\end{enumerate}
\end{definition}

We note that to verify the PHP property it suffices to choose the finite sets $F$ from a set $\mathcal S$ whose linear span is dense in $\ker \rho$. 
A further reduction is the following:

\begin{lemma}
	If we can verify the PHP property for  $n=3$, then we can verify it for all $n\in \N$. 
\end{lemma}
\begin{proof}
	Note that if property PHP holds for some $n$, it also holds for all $m\leq n$, by simply discarding some of the
	$P_i,P_i^+,u_i$. Thus, we must show that it holds for arbitrarily large $n$.
	
	Let $F\subseteq \ker \rho$ be a finite set 	and $\epsilon>0$.
	Suppose that we have  $u_i,P_i,P_i^+$ for $i=1,2,3$ that satisfy  (1), (2), (3) in the definition of the PHP property. 
	Since $P_j\leq 1-P_i$ for $i\neq j$, we have by (2) that $u_iP_ju_i^*\leq P_i^+$ for $i\neq j$. Consider 
	the family of six projections $P_{i,j}=u_iP_ju_i^*$, for $i\neq j$. Observe that they are pairwise orthogonal, since
	if $i\neq i'$ then $P_{i,j}$ and $P_{i',j'}$ are subprojections of $P_i$ and $P_{i'}$, respectively, which are orthogonal, while
	if $j\neq j'$, then $P_{i,j}$ and $P_{i,j'}$ are conjugates of $P_j$ and $P_{j'}$ by the same unitary. The norm bounds on the cut-downs $P_{i,j}xP_{i',j'}$ continue to  hold, since the new projections are subprojections of the original projections. This shows (3) for the new family of projections. Finally, to check  (2), choose $P_{i,j}^+=u_iP_j^+u_i^*$ and  $u_{i,j}=u_iu_ju_i^*$. More generally, the same procedure allows us to pass from  PHP for  $n$ to PHP for $n^2-n$.
    Thus starting with $n=3$, we obtain the PHP for arbitrarily large values of $n$. 
\end{proof}

The following holds as in \cite[Theorem~14]{ozawa2025proximalityselflessnessgroupcalgebras}. We briefly sketch the proof.
\begin{theorem}
	If $B\subset (A,\rho) \subset \mathbb B(H)$ has property PHP, then $B\subset (A,\rho)$ is a selfless inclusion.
\end{theorem}

\begin{proof}
	Let $F\subset \ker(\rho)$ be a finite set, $\epsilon>0$, and $n\in\N$. Choose $P_i,P_i^+,u_i$ for $i=1,\ldots,n$ that satisfy (1), (2), (3)  of the PHP relative to  $F$ and $\epsilon$.
	Define $T \in \mathbb B(H)$ by 
	\[
	T = \frac{1}{\sqrt{2n}}\sum_{i=1}^n(P_i^+u_i + u_i^* (1-P_i^+)).
	\]
	Then 
	\[
	T+T^*=\frac{1}{\sqrt{2n}}\sum_{i=1}^n(u_i+u_i^*)\in B.
	\]
	Moreover, as argued in \cite[Theorem~14]{ozawa2025proximalityselflessnessgroupcalgebras},  
	we have
\[	
	(1- \frac{1}{2n})\leq T^*T \leq 1
	\]
and
\[	
\|T^*xT\| < \epsilon,\quad \rho(xTT^*x^*)<\epsilon\hbox{ for all }x\in F.
\]	
	By taking a suitable ultralimit, we obtain an isometry $T \in \mathbb B(H)^{\mathcal U}$ as required in the hypothesis of Lemma \ref{lem: univ-free-prod}. Thus there is an embedding $A * \mathcal T$ into $\mathbb B(H)^{\mathcal U}$ which restricts to the diagonal embedding on $A$.  This embedding maps $T+T^*$ to a self-adjoint operator in $B^{\mathcal U}$, witnessing the selflessness of $B\subset (A,\rho)$.
\end{proof}

It is proved in  \cite{fmmm2025selflessfreeprod} that  reduced free products satisfying an Avitzour-type condition are selfless. In the theorem below we revisit the same argument, which originates in \cite{ozawa2025proximalityselflessnessgroupcalgebras}, explicitly framing it as a verification of the PHP property. 
We have tailored  the statement of this theorem to suit our application to free complexifications in the next section.

\begin{theorem}\label{alaavitzour}
	Let $(A,\rho)=(A_1,\rho_1)*(A_2,\rho_2)$ be a reduced free product. 
    Let $z\in A_1$ and $a,b,c\in A_2$ be unitaries satisfying that
	\begin{itemize}
	\item 
	$\rho_1(z)=\rho_2(a)=\rho_2(b)=\rho_2(c)=\rho_2(c^{-1}b)=\rho_2(c^{-1}a)=0$,
	\item 
	$z$ belongs to the centralizer of $\rho_1$, and $a,b,c$ belong to the centralizer of $\rho_2$.
\end{itemize}
Let $B=C^*(z^{-1}az,b,c)$. Then $B\subset (A,\rho)$ has the PHP property, and consequently it is a selfless inclusion.
\end{theorem}


Before proving this theorem we make some preliminary remarks on reduced words in reduced free products.
Let $(A,\rho)=(A_1,\rho_1)*(A_2,\rho_2)$ be as in the theorem, and identify $A_1$ and $A_2$ as C*-subalgebras of $A$.
We call $w\in A$ an alternating centered word, or a reduced word,  if 
\[
w=c_1c_2\cdots c_l
\]
where $c_i\in A_{j(i)}\ominus \C 1$ for all $i$ and $j\colon \{1,\ldots,l\}\to \{1,2\}$
is an alternating function; i.e., $j(i)\neq j(i+1)$ for all $1\leq i\leq l-1$. If $w$ is a reduced word, then its
representation as an alternating centered product as above is unique. The number of factors (i.e., letters) in this representation is called the length of the word. From the
construction of the reduced free product we have that $A$ is spanned by $1$ and the reduced words. In the proof below we shall use the following facts, which are easily established by induction. 

\begin{enumerate}
\item[Fact 1:]
Let $w,w'\in A$ be reduced words. Then $w\perp w'$ (where 
$\langle w,w'\rangle:=\rho((w')^*w)$) if $w$ and $w'$ have different length, while if they have the same length, and say $w=c_1c_2\cdots c_l$, $w'=c_1'c_2'\cdots c_l'$, then
\[
\langle w,w'\rangle =\langle c_1,c_1'\rangle\cdots \langle c_l,c_l'\rangle.
\] 
We conclude that if two reduced words have orthogonal letters at the same position, then they are orthogonal.

\item[Fact 2:] 
If $w_1,x,w_2$ are reduced words such that $w_1$ and $w_2$ are at least as long
as $x$, then 	every non-scalar summand in the reduced expansion of $w_1 x w_2$ is of one of the following forms:
	\begin{itemize}
		\item a reduced word $w_1' y w_2'$,
		where $w_1'$ is an initial segment of $w_1$, $w_2'$ is a terminal segment of $w_2$, and $y$ is a reduced word (possibly empty),
		
		\item an initial segment of $w_1$,
		
		\item a terminal segment of $w_2$.
\end{itemize}
\end{enumerate}

\begin{proof}[Proof of Theorem \ref{alaavitzour}]
	We set  $H=L^2(A,\rho)$ and identify $A$ with a C*-subalgebra of $\mathbb B(H)$ via the GNS representation.
	
For a reduced word $w\in A$, denote by $B_w\subseteq H$ the  closed linear span of the reduced words that start with $w$, i.e., of the form $w\eta$ for some reduced word $\eta$. Denote by  $P_w$  the orthogonal projection onto $B_w$.

Let $F\subseteq A\ominus\C 1$ be a finite set and $\epsilon>0$. As remarked above, to verify the PHP property it suffices to let  $F$ range through finite subsets of a set with dense linear span in $A\ominus \C 1$. Thus, we may assume that $F$ consists of reduced words. We can also simultaneously conjugate the elements of $F$ by a unitary leaving $A\ominus \C 1$ invariant (i.e., in the centralizer of $\rho$) and 
work with this new set.  Observe that, given a reduced word $x$, the element
 \[
 x'=(cz^{-1}az)^{-N_0}x(cz^{-1}az)^{N_0}\] 
 is a linear combination of nontrivial reduced words, by the invariance of the reduced free product state $\rho$ under conjugation by $z^{\pm 1}$ and $a,c$. Moreover, using Fact 2 stated above, we can choose a large enough $N_0$  such that $x'$ is a linear combination of reduced words that either
	\begin{itemize}
		\item 
		start in $z^{-1}$ and end in $z$, or
		\item 
		have either the  forms $(cz^{-1}az)^l$ or  $az(cz^{-1}az)^l$ for $l>0$,
        \item 
        are  inverses of the words in the previous bullet point.        
	\end{itemize}	
Let  $N_0$ be large enough so that the above applies to every $x\in F$. Replacing $F$ by $(cz^{-1}az)^{-N_0}F(cz^{-1}az)^{N_0}$, and then replacing this new set by the reduced words supporting its elements,  we may (and will) assume that $F$ consists of finitely many words as in the three bullet points above.

	To verify the PHP property, we set 
	\[
	w_i=(bz^{-1}az)^{i}(cz^{-1}az)(bz^{-1}az)^{3-i},\quad w_i^+=w_i(bz^{-1}az)
	\]
	for $i=1,2,3$.
	We then choose 
	\[
	P_i=P_{w_i},\quad P_i^+=P_{w_i^+},\quad u_i=w_i^+c^{-1}w_i^*,
	\]
	for $i=1,2,3$ (by a previous lemma, it suffices to verify the PHP with $n=3$). Note that
	$u_i\in C^*(z^{-1}az,b,c)=B$ for all $i$.

	To check the orthogonality between the projections $P_1,P_2,P_3$, 
we note that the first appearance of the letter $c$ in each word $w_1,w_2,w_3$ is matched at the same position with a $b$ in the other words. Since $b\perp c$, we deduce from Fact 1 stated above that the words starting in $w_i$ and $w_j$ for $i\neq j$ are orthogonal, i.e., $B_{w_i}\perp B_{w_j}$.	This proves (1) of the PHP property.

		To prove  (2) of the PHP property,
	fix $1\leq i\leq 3$ and set $w=w_i$, $w^+=w_i^+$, and $u_w=u_i$.  
	The second property in the definition of the PHP can be reformulated as 
	$u_wB_{w^+}\subseteq B_{w^+}$ and $u_{w}(H\ominus B_w)\subseteq B_{w^+}$.

	To prove $u_wB_{w^+}\subseteq B_{w^+}$ it suffices to do so on the reduced words
	spanning $B_{w^+}$. Let $\xi=w^+\eta$ be a reduced word in $B_{w^+}$. Then 
	\begin{align*}
		u_w\xi  &= (w^+c^{-1}w^*)w^+\eta\\
		&=(w^+c^{-1}w^*)w(bz^{-1}az)\eta \\
		&= w^+(c^{-1}b)z^{-1}az\eta.
	\end{align*}
	Notice that $w^+(c^{-1}b)$ is a reduced word ending in $c^{-1}b\in A_2\ominus \C$ (since $w^+$ ends in $A_1$ and  $b\perp c$). So  
	$w^+(c^{-1}b)z^{-1}az\eta$ is already reduced and starts with $w^+$.
	This verifies the first inclusion.

	Consider now the inclusion $u_{w}(H\ominus B_w)\subseteq B_{w^+}$. Again, 
	to prove this inclusion it suffices to do so on the reduced words
	spanning $H\ominus B_w$, i.e., reduced words not starting in $w$. 
	Take a reduced word $\xi\in H\ominus B_w$. Then 
	\[
	u_w\xi = w^+c^{-1}w^*\xi.
	\]
	We claim that $w^*\xi$ belongs to the span of reduced words that start in $A_1\ominus \C$.
	To show this, it suffices to check that it is orthogonal to every reduced  word that starts in $A_2\ominus \C$ and to 1.
	Let $\xi'$ be either $1$ or a reduced centered word that starts in $A_2\ominus \C$. Since $w$ ends in $z$, $w\xi'\in B_w$, and so
	$\langle w^*\xi,\xi'\rangle=\langle \xi, w\xi'\rangle=0$. The claim follows. It now follows that $w^+c^{-1}(w^*\xi)$ belongs to the span of reduced words that
	start in $w^+$, as desired.

		Finally, let us prove (3) in the definition of the PHP property, i.e., that $P_{w_i}xP_{w_j}=0$ for all $i,j$ and all $x\in F$. Note that $P_{w_i}xP_{w_j}=0$  is equivalent to
	\begin{equation}\label{perpij}
	 x w_i\eta\perp w_j\eta'
	\end{equation}
	for all $i,j$ and reduced words $w_i\eta$ and $w_j\eta'$. Suppose first that $x$ is a reduced word starting in $z^{-1}$ and ending in 
	$z$. Then $xw_i$ is already  a reduced word starting in $z^{-1}\in A_1\ominus \C$, while $w_j\eta'$ starts in $b\in A_2\ominus \C$. We thus have the desired orthogonality \eqref{perpij}.
Suppose that  $x=(cz^{-1}az)^l$ for some nonzero $l>0$. Then
$xw_i\eta$ is a reduced word that starts in $c$, while
$w_j\eta'$ starts in $b$. Since $b\perp c$, we again have \eqref{perpij}. Suppose that $x=az(cz^{-1}az)^l$ for $l>0$.
Then $xw_i\eta$ is a reduced word whose third letter is a $c$, matched in the position by an $a$ in $w_j\eta'$. We 
thus have \eqref{perpij}. Finally, the 
inverses of these two types of words can be dealt with by rewriting  
\eqref{perpij} as $w_i\eta\perp x^{-1}w_j\eta'$.
\end{proof}

\section{Free complexifications and selfless compact quantum groups}
Let us  discuss the operation of free complexification, introduced by Banica \cite{BanicaNote2008}, and further studied in \cite{Raum2012}, \cite{TarragoWeber}. 
Here we work with reduced free products rather than universal ones, in contrast to \cite{BanicaNote2008}.

Let $(A,\rho)$ be a C*-probability space.  Let $X\subseteq A$ be a set generating $A$ as a unital C*-algebra; i.e., such that
$A=C^*(X, 1)$. Form the reduced free product $(A,\rho)*(C(\T),\lambda)$ (where $\lambda$ denotes the state induced by the Lebesgue measure).
Let 
\[
Xz:=\{xz:x\in X\}\subseteq A*C(\T),
\] 
where $z$ is the canonical generator of $C(\T)$. We call
\[
\widetilde A =C^*(1,Xz)\subseteq A*C(\T)
\]
the free complexification of $(A,\rho,X)$. We endow $\widetilde A$ with the state $(\rho*\lambda)\big|_{\widetilde A}$ and the generating set $Xz$.
Note that, although not explicitly reflected in our notation,  the free complexification $\widetilde A$ depends  
both on $\rho$ and $X$. 

We denote by $\PV_X A\subseteq A$ the C*-subalgebra generated by $\{1\}$
and 	
\[
XX^* = \{x y^* : x,y\in X\}.
\]
We call $(\PV_X A,XX^*)$ the projective version
of $(A,X)$ (see \cite{BanicaIntro}). Note that, although it is not generally the  case that $A\subseteq \widetilde A$,
we do have $\PV_X A\subseteq \widetilde A$.
In the theorem below  we also make reference to  the C*-subalgebra $\PV_{X^*} A$ generated by $\{1\}$
and 	$X^*X$; i.e., the projective version of $(A,X^*)$. 
We note that $z^{-1}(\PV_{X^*}A) z\subseteq \widetilde A$.


\begin{theorem}\label{thm:freecomplexification}
Let $(A,X,\rho)$ be a C*-probability space endowed with a generating set $X$. Suppose that there exist unitaries 
$a\in \PV_{X^*}A$ and $b,c\in \PV_{X}A$ in the centralizer of $\rho$ and such that $\rho(a)=\rho(b)=\rho(c)=\rho(c^{-1}a)=\rho(c^{-1}b)=0$. Then the inclusion
$\widetilde A\subset A*C(\T)$ is selfless. In particular, the free complexification $\widetilde A$ is selfless.	
\end{theorem}		
	
\begin{proof}	
Consider the reduced free product $(A,\rho)*(C(\T),\lambda)$. We readily check that the hypotheses of Theorem \ref{alaavitzour} are verified 
for this free product with $A_1=C(\T)$, $A_2=A$, and 
the unitaries $z\in C(\T)$ and $a,b,c\in A$. Thus, defining $B=C^*(z^{-1}az,b,c)$, we deduce that
the inclusion $B\subset A*C(\T)$ is selfless. On the other hand, it is easily checked that $z^{-1}az,b,c\in \widetilde A$. Thus, the inclusion
$\widetilde A\subset A*C(\T)$ is selfless.
\end{proof}	

Let us recall the definition of the reduced free unitary and free orthogonal compact quantum groups. 

Let $n\in\N$, with $n\geq 2$.  The reduced free unitary compact quantum group is obtained as follows: start from the universal C*-algebra $\fA$ generated
by elements $\{u_{i,j}:i,j=1,\ldots,n\}$ such that, with $u=(u_{ij})_{ij}$, the relations
\[
u^*u=uu^*=1_n,\quad (u^t)^*(u^t)=u^t(u^t)^*=1_n 
\]	
are satisfied (where $u^t$ is the transpose of $u$).
This C*-algebra comes equipped with a comultiplication $\Delta:\fA \to \fA\otimes_{\min}\fA$ defined by $\Delta(u_{ij}) = \sum_k u_{ik}\otimes u_{kj}$. It also comes with a unique Haar state $\tau$ which is tracial and invariant under $\Delta$; i.e., for all $a\in\fA$, $(\tau\otimes  \iota)\Delta(a) = (\iota\otimes \tau)\Delta(a) = \tau(a)$ (see \cite{Woro1987cptquantumgps,BanicaCollins2007cptquantumgpsHaar}). Take the GNS representation of this universal C*-algebra with respect to the Haar trace. The resulting C*-algebra is the  reduced free unitary compact quantum group. We denote it by $A_u(n)$.

Similarly, the reduced free orthogonal compact quantum group is obtained starting from the universal C*-algebra generated by elements $\{v_{ij}\}_{i,j=1}^n$  satisfying
\[
v^*v=vv^*=1_n,\qquad v_{ij}^*=v_{ij}\ \text{ for all } i,j,
\]
with $v=(v_{ij})$ (which again comes equipped with the comultiplication $\Delta(v_{ij}) = \sum_k v_{ik}\otimes v_{kj}$) and passing to the GNS representation with respect to its (tracial) Haar state. We denote the reduced free orthogonal compact quantum group
by $A_o(n)$.

Banica showed \cite{BanicaU_n1997} (see also \cite[Theorem~3.4]{BanicaNote2008}),  that 	$A_u(n) \cong \widetilde{A_o(n)}$,
for $n\ge 2$. More concretely,  $A_u(n)$ is isomorphic to the free complexification of $A_o(n)$, relative to  the Haar state and generating set
$X$ given by the entries of $v=(v_{i,j})_{ij}$.  
For $n=2$, Banica showed in \cite{BanicaU_n1997} that  $A_u(2)$ can be alternatively obtained as the free complexification of $(C(SU(2)),\lambda, \{v_{ij}\}_{ij})$, where in this case $\lambda$ is given by the Haar measure on $SU(2)$ and 
$v\in M_2(C(SU(2)))$ is the fundamental representation of $SU(2)$; i.e., the inclusion map 	$SU(2)\to M_2(\C)$.

	\begin{theorem}\label{thm:cpct-qntm-unitry-slflss}
		$A_u(n)$ is selfless for all $n\ge2$.
	\end{theorem}
	
	\begin{proof}
First assume $n>2$. As remarked above, $A_u(n)$ is the free complexification of $(A_o(n),\tau)$ relative to the generating set
$X=\{v_{ij}:i,j\}$. Since $X$ is a selfadjoint set, $\PV_{X} A_o(n)=\PV_{X^*} A_o(n)$. 
We will be done once we have shown that $\PV_X A_o(n)$ contains a Haar unitary $u$, as then we can 
apply Theorem \ref{thm:freecomplexification} with $a=b=u$ and $c=u^2$.
It is known from \cite[Proposition 1]{BanicaU_n1997} that $s=\sum_{i=1}^n v_{ii}$ has a semicircular distribution. It follows that 
$|s|:=\sqrt{s^2}\in \PV_X A_o(n)$ has a quarter-circular distribution (which is diffuse). Thus, through functional calculus, we can find a Haar unitary in $\PV_X A_o(n)$.

Assume that $n=2$. In this case $A_u(2)$ is the free complexification of $(C(SU(2)), \lambda)$, where the generating set $X$
is the entries of the fundamental representation $v\in M_2(C(SU(2)))$ and $\lambda$ is induced by the Haar measure on $SU(2)$. 
We have 
\[
\PV_{X^*}  C(SU(2))=\PV_{X} C(SU(2))\cong C(SO(3)),
\]
where the first equality holds  by the commutativity of  $C(SU(2))$ and the second one by  \cite[Th\'eor\`eme 5]{BanicaU_n1997}.
Since the Haar measure on $SO(3)$ is diffuse, $C(SO(3))$ contains a Haar unitary $u$. As in the previous case,  choosing $a=b=u$ and $c=u^2$, we can apply Theorem \ref{thm:freecomplexification} to reach the desired conclusion.
\end{proof}

Another class of compact quantum groups that arise as free complexifications are the half-liberated unitary quantum groups introduced in \cite{BDDstandardmodel}. We briefly recall the definition. For $n\in\mathbb N$, the half-liberated unitary compact quantum group is the universal C*-algebra generated by the entries of an $n\times n$ matrix $u$, subject to the same relations as the free unitary compact quantum group—namely, that $u$ and $u^t$ are unitaries—and, in addition, the relations
\[
ab^*c=cb^*a\qquad\text{for all }a,b,c\in\{u_{ij}:i,j=1,\ldots,n\}.
\]
We denote by $A_u^*(n)$ the reduced half-liberated quantum unitary group, i.e., the image of the universal one under the GNS representation induced by the Haar state.

\begin{corollary}
The reduced half-liberated unitary compact quantum groups $A_u^*(n)$ are selfless for $n\ge 2$.
\end{corollary}

\begin{proof}
By \cite[Proposition 3.7]{BanicaBichon}, $A_u^*(n)$ is isomorphic to the free complexification of $(C(U(n)),\tau)$, where 
$U(n)$ is the group of $n\times n$ unitary matrices, $\tau$ is the Haar state, and
the generating set $X$ is the entries of the fundamental representation of $U(n)$. 
 (Note: In \cite{BanicaBichon}, the universal half-liberated unitary quantum group is denoted by $C(U_n^\times)$, so  
$A_u^*(n)$ is the reduced version of $C(U_n^\times)$.) The projective versions $\PV_X C(U(n))$ and $\PV_{X^*}C(U(n))$ agree and are isomorphic to $C(PU(n))$, where $PU(n):=U(n)/\mathbb T$ (see again \cite[Proposition 3.7]{BanicaBichon}). Since the Haar measure on $PU(n)$ is diffuse for $n\ge 2$, we can find a Haar unitary  $u\in C(PU(n))$.
Hence, the hypotheses of Theorem \ref{thm:freecomplexification} are satisfied with $a=b=u$ and $c=u^2$.
\end{proof}

For further examples of compact quantum groups obtained as free complexifications see \cite[Theorem~3.4]{BanicaNote2008}, \cite{Raum2012}. 	There is  also a construction of free $d$-complexification where the reduced free product with
$C(\T)$ is replaced with a reduced free product with $\C^d$. Examples of compact quantum groups obtained as free $d$-complexifications are obtained in 
\cite[Theorem~6.16]{TarragoWeber}.

\bibliographystyle{alpha}
	\bibliography{references}

\newcommand{\etalchar}[1]{$^{#1}$}
\begin{thebibliography}{AGKEP25}

\bibitem[AAP86]{AAPdiffuse}
Charles~A. Akemann, Joel Anderson, and Gert~K. Pedersen.
\newblock Diffuse sequences and perfect {$C^\ast$}-algebras.
\newblock {\em Trans. Amer. Math. Soc.}, 298(2):747--762, 1986.

\bibitem[AGKEP25]{amrutam2025strictcomparisonreducedgroup}
Tattwamasi Amrutam, David Gao, Srivatsav Kunnawalkam~Elayavalli, and Gregory Patchell.
\newblock Strict comparison in reduced group {$\rm C^*$}-algebras.
\newblock {\em to appear in Invent. Math}, 2025.

\bibitem[Avi82]{Avitzour}
Daniel Avitzour.
\newblock Free products of {$C\sp{\ast} $}-algebras.
\newblock {\em Trans. Amer. Math. Soc.}, 271(2):423--435, 1982.

\bibitem[Ban97]{BanicaU_n1997}
Teodor Banica.
\newblock Le groupe quantique compact libre {${\rm U}(n)$}.
\newblock {\em Comm. Math. Phys.}, 190(1):143--172, 1997.

\bibitem[Ban08]{BanicaNote2008}
Teodor Banica.
\newblock A note on free quantum groups.
\newblock {\em Ann. Math. Blaise Pascal}, 15(2):135--146, 2008.

\bibitem[Ban22]{BanicaIntro}
Teodor Banica.
\newblock {\em Introduction to quantum groups}.
\newblock Springer, Cham, [2022] \copyright 2022.

\bibitem[BB17]{BanicaBichon}
Teodor Banica and Julien Bichon.
\newblock Complex analogues of the half-classical geometry.
\newblock {\em M\"unster J. Math.}, 10(2):457--483, 2017.

\bibitem[BC07]{BanicaCollins2007cptquantumgpsHaar}
Teodor Banica and Beno\^it Collins.
\newblock Integration over compact quantum groups.
\newblock {\em Publ. Res. Inst. Math. Sci.}, 43(2):277--302, 2007.

\bibitem[BDD11]{BDDstandardmodel}
Jyotishman Bhowmick, Francesco D'Andrea, and Ludwik Dabrowski.
\newblock Quantum isometries of the finite noncommutative geometry of the standard model.
\newblock {\em Comm. Math. Phys.}, 307(1):101--131, 2011.

\bibitem[BO08]{brown2008textrm}
Nathanial~Patrick Brown and Narutaka Ozawa.
\newblock {\em {$\rm C^*$}-Algebras and Finite-Dimensional Approximations}, volume~88.
\newblock American Mathematical Soc., 2008.

\bibitem[BS16]{BarlakSzabo2016seqsplit}
Sel\c{c}uk Barlak and G\'abor Szab\'o.
\newblock Sequentially split {$\ast$}-homomorphisms between {$\rm C^*$}-algebras.
\newblock {\em Internat. J. Math.}, 27(13):1650105, 48, 2016.

\bibitem[FHL{\etalchar{+}}21]{FHLRTVW2021modelthryCstar}
Ilijas Farah, Bradd Hart, Martino Lupini, Leonel Robert, Aaron Tikuisis, Alessandro Vignati, and Wilhelm Winter.
\newblock Model theory of {$\rm C^*$}-algebras.
\newblock {\em Mem. Amer. Math. Soc.}, 271(1324):viii+127, 2021.

\bibitem[FK{\'O}CP25]{fmmm2025selflessfreeprod}
Felipe Flores, Mario Klisse, M{\'i}che{\'a}l {\'O}~Cobhthaigh, and Matteo Pagliero.
\newblock Selfless reduced free products and graph products of {$\rm C^*$}-algebras.
\newblock {\em Preprint}, 2025.

\bibitem[Hay15]{Hayes3}
Ben Hayes.
\newblock An {$l^p$}-version of von {N}eumann dimension for {B}anach space representations of sofic groups {II}.
\newblock {\em J. Funct. Anal.}, 269(8):2365--2426, 2015.

\bibitem[HKER25]{HKER}
Ben Hayes, Srivatsav Kunnawalkam~Elayavalli, and Leonel Robert.
\newblock Selfless reduced free product {$\rm C^*$}-algebras.
\newblock {\em https://arxiv.org/abs/2505.13265}, 2025.

\bibitem[Jun05]{junge2005picuineq}
Marius Junge.
\newblock Embedding of the operator space {$OH$} and the logarithmic `little {G}rothendieck inequality'.
\newblock {\em Invent. Math.}, 161(2):225--286, 2005.

\bibitem[KES25]{elayavalli2025negativeresolutioncalgebraictarski}
Srivatsav Kunnawalkam~Elayavalli and Christopher Schafhauser.
\newblock Negative resolution to the {$\rm C^*$}-algebraic {Tarski} problem.
\newblock {\em https://arxiv.org/abs/2503.10505}, 2025.

\bibitem[Mal12]{Male}
Camille Male.
\newblock The norm of polynomials in large random and deterministic matrices.
\newblock {\em Probab. Theory Related Fields}, 154(3-4):477--532, 2012.
\newblock With an appendix by Dimitri Shlyakhtenko.

\bibitem[Oza23]{ozawa2023amenability}
Narutaka Ozawa.
\newblock Amenability for unitary groups of simple monotracial {$\rm C^*$}-algebras.
\newblock {\em arXiv preprint arXiv:2307.08267}, 2023.

\bibitem[Oza25]{ozawa2025proximalityselflessnessgroupcalgebras}
Narutaka Ozawa.
\newblock Proximality and selflessness for group {$C^*$}-algebras, 2025.

\bibitem[Rau12]{Raum2012}
Sven Raum.
\newblock Isomorphisms and fusion rules of orthogonal free quantum groups and their free complexifications.
\newblock {\em Proc. Amer. Math. Soc.}, 140(9):3207--3218, 2012.

\bibitem[Rob25]{robertselfless}
Leonel Robert.
\newblock Selfless {${\rm C}^*$}-algebras.
\newblock {\em Adv. Math.}, 478:Paper No. 110409, 28, 2025.

\bibitem[R{\o}r23a]{Rordam-Avg2023}
Mikael R{\o}rdam.
\newblock A dixmier type averaging property of automorphisms on a {$C^*$}-algebra.
\newblock {\em International Journal of Mathematics}, 34(4):2350017, 2023.
\newblock 20 pp.

\bibitem[R{\o}r23b]{RordamIrred}
Mikael R{\o}rdam.
\newblock Irreducible inclusions of simple {C}*-algebras.
\newblock {\em Enseign.Math.}, 69(2):275–314, 2023.

\bibitem[RS02]{RordamStormer-EMS126}
Mikael R{\o}rdam and Erling St{\o}rmer.
\newblock {\em Classification of Nuclear C*-Algebras. Entropy in Operator Algebras}, volume 126 of {\em Encyclopaedia of Mathematical Sciences}.
\newblock Springer-Verlag, Berlin Heidelberg, 2002.

\bibitem[RTV25]{raum2025strictcomparisontwistedgroup}
Sven Raum, Hannes Thiel, and Eduard Vilalta.
\newblock Strict comparison for twisted group {$\rm C^*$}-algebras.
\newblock {\em https://arxiv.org/abs/2505.18569}, 2025.

\bibitem[Sar25]{sarkowicz2025tensorially}
Pawel Sarkowicz.
\newblock Tensorially absorbing inclusions of {$\rm C^*$}-algebras.
\newblock {\em Canad. J. Math.}, 77(4):1315--1346, 2025.

\bibitem[TW17]{TarragoWeber}
Pierre Tarrago and Moritz Weber.
\newblock Unitary easy quantum groups: the free case and the group case.
\newblock {\em Int. Math. Res. Not. IMRN}, 2017(18):5710--5750, 2017.

\bibitem[vH25]{vHCDM}
Ramon van Handel.
\newblock The strong convergence phenomenon, 2025.

\bibitem[Vig25]{vigdorovich2025structuralpropertiesreducedcalgebras}
Itamar Vigdorovich.
\newblock Structural properties of reduced {$\rm C^*$}-algebras associated with higher-rank lattices.
\newblock {\em https://arxiv.org/abs/2503.12737}, 2025.

\bibitem[Wor87]{Woro1987cptquantumgps}
S.~L. Woronowicz.
\newblock Compact matrix pseudogroups.
\newblock {\em Comm. Math. Phys.}, 111(4):613--665, 1987.

\end{thebibliography}

\end{document}